\documentclass[]{amsart}
\usepackage{amsfonts}
\usepackage{amssymb}
\usepackage{nicefrac}
\usepackage{enumerate}
\usepackage{url,bm}
\usepackage{color}      % Need the color package%
\usepackage{capt-of}
\usepackage{graphicx}

\newcommand{\R}{\mathbb{R}}
\newcommand{\N}{\mathbb{N}}
\newcommand{\C}{\mathbb{C}}
\newcommand{\Z}{\mathbb{Z}}
\newcommand{\F}{\mathcal{F}}
\newcommand{\A}{\mathcal{A}}
\newcommand{\V}{\mathcal{V}}

\newcommand{\T}{\mathcal{T}}%Tau

\newcommand{\Cc}{\mathcal{C}}
\newcommand{\Hh}{\mathcal{H}}

\newcommand{\U}{\mathcal{U}}
\newcommand{\w}{\omega}
\newcommand{\si}{\sigma}
\newcommand{\la}{\lambda}

\newtheorem{theorem}{Theorem}[section]
\newtheorem{lemma}[theorem]{Lemma}

\newtheorem{proposition}[theorem]{Proposition}

\newtheorem*{general-problem}{GP }
\newtheorem*{Problema A}{Problem A}
\newtheorem*{Problema B}{Problem B}

\newtheorem*{Problema C}{Problem C}
\theoremstyle{remark}
\newtheorem{remark}[theorem]{Remark}

\theoremstyle{definition}
\newtheorem{definition}[theorem]{Definition}

\date{January 2019}
\thanks{The research of the authors is partially supported by Grants: CONICET PIP 11220110101018, PICT 2011-436
and UBACyT 20020130100422BA\\Corresponding author Carlos Cabrelli, e-mail: cabrelli@dm.uba.ar,
Telephone/Fax +54 11 4788 1824}

\begin{document}
\title[Subspaces with extra invariance nearest to observed data]{Subspaces with extra invariance nearest to observed
data.}

\author[ C. Cabrelli and C. A. Mosquera]{C. Cabrelli and  C. A. Mosquera}

\address{\textrm{(C. Cabrelli)}
Departamento de Matem\'atica,
Facultad de Ciencias Exac\-tas y Naturales,
Universidad de Buenos Aires, Ciudad Universitaria, Pabell\'on I,
1428 Buenos Aires, Argentina and
IMAS-CONICET, Consejo Nacional de Investigaciones
Cient\'ificas y T\'ecnicas, Argentina}
\email{cabrelli@dm.uba.ar}

\address{\textrm{(C. A. Mosquera)}
Departamento de Matem\'atica,
Facultad de Ciencias Exac\-tas y Naturales,
Universidad de Buenos Aires, Ciudad Universitaria, Pabell\'on I,
1428 Buenos Aires, Argentina and
IMAS-CONICET, Consejo Nacional de Investigaciones
Cient\'ificas y T\'ecnicas, Argentina}
\email{mosquera@dm.uba.ar}

\begin{abstract}

Given an arbitrary finite set of data $\F= \{f_1, \dots, f_m\}\subset L^2(\R^d)$ we prove the existence and show how to construct
a  ``small shift invariant space'' that is  ``closest'' to the data $\F$ over certain class of closed subspaces of $L^2(\R^d)$. The approximating subspace is required to have extra-invariance properties, that is to be invariant under translations by a prefixed additive subgroup of $\R^d$ containing $\Z^d$. This is important for example in situations where we need to deal with jitter error of the data.
Here small means that our solution subspace should be generated by the integer translates of a small number of generators.
An expression for the error in terms of the data is provided and we construct a Parseval frame for the optimal space.

We also consider the problem of approximating  $\F$ from generalized Paley-Wiener spaces of $\R^d,$ that are  generated
by the integer translates of a finite number of functions. These spaces can be seen as finitely generated shift invariant spaces that are $\R^d$ invariant.
 
 In addition, we  characterize these spaces in terms of multi-tile sets of $\R^d$, and show the connections with recent results on Riesz basis of exponentials on bounded sets of $\R^d.$.
\end{abstract}

\maketitle
%%%%%%%
\section{Introduction}

Let $\Hh$ be a Hilbert space, $\Cc$ a class of closed subspaces of $\Hh$ and $\F= \{f_1, \dots, f_m\}$  a finite set of elements in $\Hh$. 

In this article we study the existence and show how to construct  an optimal subspace  $\mathcal{S}$ in the class $\Cc$ that minimizes the distance to the given data $\F$, in the sense that $\mathcal{S}$ 
minimizes the functional $\mathcal{E}(\F,\mathcal{S})$ over $\Cc.$ The functional is defined as
\begin{equation}\label{error}
\mathcal{E}(\F,\mathcal{S})= \sum_{j=1}^{m} \|f_j-P_{\mathcal{S}} f_j\|^2,
\end{equation}
where $P_{\mathcal{S}}$ denotes the orthogonal projection on the subspace $\mathcal{S}.$

The motivation to find an optimal subspace in $\Cc$ is, that in many situations one wants to choose a model for a certain
class of data. Instead of imposing some conditions on the data to fit some known model, the idea is to
define a large class of subspaces convenient for the application at hand, and find from there the one that
``best fits" the data under study. 

The signals that need to be modelled are ideally  low dimensional but living in a high dimensional space. However, since in applications they are  often corrupted by noise, they become high dimensional, however they are close to a low dimensional subspace, which is the space one seeks.

When the Hilbert space is $L^2(\R^d)$ it is natural to consider as a model for our data the class of
shift invariant spaces (SIS), that is,  closed subspaces of $L^2(\R^d)$ that are invariant under  translations by integers. 
These spaces have been used in approximation theory, 
harmonic analysis, wavelet theory, sampling theory and signal processing (see, e.g., \cite{AG01, Gro01, HW96, Mal89} and references therein). 
Often, in applications, it is assumed  that the signals under study belong to some shift invariant space $V$ generated by the translations of a finite set of functions $\Phi=\{\varphi_1,\cdots,\varphi_m\},$
i.e., $V = S(\Phi)= \overline{\mbox{span}}\{T_k\varphi_i\colon k\in\Z^d, i=1,\cdots, m\}.$ 

The choice of the particular finitely generated shift invariant space typically is not deduced from a set of signals.
For example in  sampling theory,  a  classical assumption is that the signals to be sampled are band-limited, that is, they belong to  
 the shift invariant space $V$  generated by $\varphi(x)= \mbox{sinc}(x).$   However, the band-limited assumption is not  very realistic in many applications.
Thus, it is natural to search for a 
finitely generated shift invariant space that is nearest to a set of some observed data. 

%%%%%%%%%%%%%%%%%%%%%%%

In this paper we study the case when $\Hh = L^2(\R^d).$
 For this case, we restrict the class of approximating subspaces to be shift invariant spaces that have extra-invariance, that is:
If $M$ is a subgroup of $\R^d$  such that $\Z^d \subset M$ we will say that 
$S(\varphi_1,\cdots,\varphi_m)$ is $M$ extra-invariant if
$$
\overline{\mbox{span}}\{T_k\varphi_j\colon j=1,\dots,m,\; k\in\Z^d\}=\overline{\mbox{span}}\{T_\alpha\varphi_j\colon j=1,\dots,m,\; \alpha\in M\}.
$$
Therefore, the space $S(\varphi_1,\cdots,\varphi_m)$   is invariant under translates other than the integers, 
even though it is  generated by the integer translates of a finite set of functions.
Such spaces with extra-invariance 
are important in applications specially in those where the  jitter error is an issue. 

We first consider the case when the subgroup $M$ is a proper subgroup of $\R^d$ that contains $\Z^d$. For that case we obtain one of the main contributions of this paper. We prove  that for any finite set of data $\F= \{f_1, \dots, f_m\}\subset L^2(\R^d),$ for any proper subgroup $M$ containing $\Z^d$ and for any $\ell \in \N$ there always exists a 
 SIS $V$ of length at most $\ell$ with extra invariance $M$ whose distance (in the sense of \eqref{error}) 
 to the data $\F$ is the smallest possible among all the SIS of  length smaller or equal than $\ell$ that are  
 $M$ extra-invariant.
(Here, the length of a SIS is the cardinal of the smallest set of generators).

We construct a solution $V$ and provide a set of generators whose integer translates form a tight frame of $V.$
An expression for the exact value of the error $\mathcal{E}(\F,V)$ between the data and the 
optimal subspace is also obtained using the eigenvalues of some special matrix.

Next, we consider the approximation problem for the class of generalized Paley-Wiener spaces.
Given a measurable set $\Omega \subset \R^d$ (not necessarily bounded), the generalized Paley-Wiener space $PW_{\Omega}$ associate to $\Omega$ is the subspace
of $L^2(\R^d)$ corresponding to the functions whose Fourier transform vanished outside $\Omega.$
A generalized Paley-Wiener space  is always invariant under translations by the whole group $\R^d.$ In particular is a SIS of $L^2(\R^d)$ that not necessarily need to be finitely generated. 
Under the hypothesis that  $PW_{\Omega}$ has a Riesz basis of integer translates, we proved that $PW_{\Omega}$ is finitely generated if and only if $\Omega$ is a multi-tile. (see Proposition \ref{propPW}).
That is  $\Omega$ is a multi-tile  if and only if $PW_{\Omega}$ has extra invariance $M=\R^d$.

We study our approximation problem for those generalized Paley-Wiener spaces. 
We describe for this case how to construct a set of generators and show an interesting connection with recent results about bases of exponentials. This complete all the cases of extra invariance when $M$ is not a proper subgroup.

Finally we consider a similar problem when  the Hilbert space is $\ell^2(\Z^d)$ and $\Cc$ is a conveniently chosen class of  subspaces. We obtain for this case equivalent results
to the ones for $L^2(\R^d).$ 
The approximation problem for the discrete case is related with the continuous case in a very interesting way
that is described in Section \ref{section0}.

\subsection{Previous Work}
Let us now mention some previous related work. 
The problem of approximation of a set of data by shift invariant spaces (without the extra invariance restriction) started in \cite{ACHM07} where the 
authors proved the existence of a minimizer for \eqref{error} over the class of low dimensional subspaces in a Hilbert space 
$\mathcal{H}$ and also over the class of shift invariant spaces in $L^2(\R^d)$.

In \cite{ACM08} the case of multiple subspaces was considered
in the finitely dimensional case. That is, the authors found a union of low dimensional subspaces  that best fits a given set of data in $\R^d$ and provided an algorithm to find it. In 2011, using dimensional reduction techniques, this algorithm was improved (see \cite{AACM11}). 

Further,  in \cite{AT11} the authors found necessary and sufficient  conditions for the existence of optimal subspaces in the general context of Hilbert spaces. However they did not provide a way to construct them.

The first result for approximation of  a finite set of data using shift invariant spaces with extra-invariance constrains appears in  \cite{AKTW12}, 
where the authors consider  principal shift invariant spaces in one variable and they assume that the space has  a generator with orthogonal integer translates, which is a key element in their proof.
So the techniques of this particular case do not apply to our general case.
 
%%%%%%%%%%%% Organization of the paper %%%%%%%%%%%%%%
%\bigskip

\subsection{Organization of the paper}
The paper is organized as follows. In Section \ref{preliminaries} we set the definitions and results that we need about shift invariant spaces, extra-invariance and 
the approximation problem for the case
of shift invariant spaces in $L^2(\R^d).$ 

The main results of the paper are stated and proved in Sections \ref{section1}, \ref{section2} and 
\ref{section0}. In Section \ref{section1} we present the  $M$ extra-invariant case for shift invariant spaces, in Section \ref{section2}  the case of Paley-Wiener spaces and finally we  consider a discrete case, in Section \ref{section0}.

%%%%%%%%%%%%%%%%%%%%%%%%%%%%%%%%%%%%%%%%%%%%%%%%%
%%%%%%%%%%%%%%%%%%%%%%%%%%%%%%%%%%%%%%%%%%%%%%%%%%
\section{preliminaries}\label{preliminaries}
We begin with a review of the basic results and definitions that will be needed in subsequent sections. The known results are generally stated without 
proofs, but we provide references where the proofs can be found. Also, we introduce some of our notational conventions. For the definitions of Riesz 
bases and frames in Hilbert spaces we refer the reader to \cite{Chr03, Hei11} and the references therein. 

%%%%%%%%%%%%%%%%%%%%%%%%%%
%%%%%%%%%%%%%%%%%%%%%%%%%%

\subsection{Shift Invariant Spaces}\label{sis}
The structure of these spaces has been deeply analyzed  (see for example \cite{Bow00, dBDVR94, dBDVRI94, Hel64, RS95}).
\begin{definition}
A closed subspace $V\subset  L^2(\R^d)$ is said to be a {\it shift invariant space} if
\[
f\in V\Longrightarrow T_kf\in V, \,\, \textrm{ for any }\,\,k\in\Z^d,
\]
where $T_k$ is the translation by the vector $k\in\Z^d,$ i.e. $T_kf(x)=f(x-k)$.

For any subset $\Phi\subset  L^2(\R^d)$ we define \[
S(\Phi)= \overline{\mbox{span}}\{T_k\varphi\colon \varphi\in\Phi, k\in\Z^d\}\,\,\textrm{ and }\,\,
E(\Phi)= \{T_k\varphi\colon \varphi\in\Phi, k\in\Z^d\}.
\]

We call $S(\Phi)$ the shift invariant space (SIS) generated by $\Phi$.
If $V=S(\Phi)$ for some finite set $\Phi$  we say that $V$ is a {\it finitely generated} SIS, and a
{\it principal} SIS if $V$ can be generated by the integer translates of a single function.
\end{definition}

For a finitely generated SIS $V\subseteq  L^2(\R^d)$ we define the length of $V$ as
\[
\ell(V)= \min\{n\in\N\colon \exists \,\,\varphi_1, \cdots,\varphi_n \in
V \textrm{ with } V=S(\varphi_1, \cdots,\varphi_n)\}.
\]
In addition to the construction of a set of generators of the optimal space for the problems considered in this paper, it will be important to estimate the error of
these approximations. In order to compute these errors we need to consider what is called the Gramian $G_{\Phi}$ for a family of functions $\Phi\subset L^2(\R^d).$

More precisely, given $\Phi= \{\varphi_1,\cdots,\varphi_m\}$ a finite collection of functions in $L^2(\R^d),$ the \emph{Gramian} $G_\Phi$ of $\Phi$ 
is the $m\times m$ matrix of $\Z^d$-periodic functions
\begin{equation} \label{gram}
[G_{\Phi}(\omega)]_{ij}
= \sum_{k \in \Z^d} \widehat{\varphi}_i(\omega+k) \, \overline{\widehat{\varphi}_j(\omega+k)}.
\end{equation}
The Gramian of $\Phi$ is determined a.e. by its values at any measurable set of representatives $\U$ of the quotient $\R^d/\Z^d$ and satisfies $G_{\Phi}(\w)^*= G_{\Phi}(\w)$ for a.e. $\w\in \U$. We will take  $\U=[-1/2,1/2)^d.$

For  a finitely generated SIS $V$, we can express the length of $V$ in terms of the Gramian as follows (see \cite{Bow00, dBDVR94, TW12})
\begin{equation}\label{length-gramian}
\ell(V)=\mathop{essup }_{\w \in \U}\big{[}\text{rk}(G_{\Phi}(\w))\big{]}
\end{equation}
where $\text{rk}(B)$ denotes the rank of a matrix $B$ and $\Phi$ is a generator set for $V$.

One important property of the Gramian is given by the following lemma concerning the measurability of the eigenvalues and the existence of measurable eigenvectors
of a non-negative matrix with measurable entries.
\begin{lemma}[Lemma 2.3.5 of \cite{RS95}]\label{medibilidad}
Let $G(\w)$ be an $m\times m$ self-adjoint matrix of measurable functions defined on a measurable subset $E\subset \R^d$ with eigenvalues $\lambda_1(\w)\ge\dots\ge \lambda_m(\w).$ Then the eigenvalues $\lambda_i,$ $i=1,\dots, m,$ are measurable functions on $E$ and there exists an $m\times m$ matrix of measurable functions $U(\w)$ on $E$ such that $U(\w)U^*(\w)=I$ a.e. $\w\in E$ and such that
\[
G(\w)=U(\w)\Lambda(\w)U^*(\w), \quad \text{ a.e. }\w\in E,
\]
where $\Lambda(\w):=\text{diag}(\lambda_1(\w),\dots, \lambda_m(\w)).$
\end{lemma}
 
In \cite{Hel64}, Helson introduced  range functions and used this notion to completely characterize shift  invariant spaces.
Later on, several authors have used this framework to describe and characterize frames and bases of these spaces.
See for example \cite{Bow00, CP10, dBDVR94, dBDVRI94, RS95}.
We will mention the required definitions and some known results that we need later for giving the proofs of our results. 
We refer to \cite{Bow00,dBDVR94, dBDVRI94, RS95} for a complete description and the proofs.
\begin{definition}\label{U}
Let $f\in L^2(\R^d)$ and fix $\U \subset \R^d$ to be a measurable set of representatives  of the quotient $\R^d/\Z^d.$
For  $\w\in \U,$ the {\it fiber} $\tau f(\w)$ of $f$ at $\w$ is the sequence
\[
\tau f(\w)= \{\widehat{f}(\w+k)\}_{k\in\Z^d}.
\]
\end{definition}
Here $\widehat{f}$ denotes the Fourier transform of the function $f,$ that is $\widehat{f}(\w)\!= \int_{\R^d}e^{-2\pi i\w x}f(x)\, dx$ when $f\in L^1(\R^d).$
We observe that if $f\in L^2(\R^d),$ then the fiber $\tau f(\w)$ belongs to $\ell^2(\Z^d)$ for almost every $\w\in \U.$

If $V$ is a finitely generated SIS and $\w\in \U$ we define the {\it fiber space} associated to $V$ and $\w$ as follows
\[
J_V(\w)=\overline{\{\tau f(\w)\colon f\in V\}},
\]
where the closure is taken in the norm of $\ell^2(\Z^d).$

With the above definitions we have:

\begin{lemma}[Proposition 5.6 of \cite{ACP11}]\label{dimension}
Let $V=S(\Phi)$ be a finitely generated SIS. Then
\[
\dim(J_V(\w))= \text{rk }(G_{\Phi}(\w)), \, \text{ a.e. }\w\in \U.
\]
\end{lemma}

\begin{lemma}\label{Lema 3.1}
If $f\in L^2(\R^d),$ then
\begin{enumerate}
\item[$(i)$] the sequence $\tau f(\w)= \{\widehat{f}(\w+k)\}_{k\in\Z^d}$ is a well-defined sequence in $\ell^2(\Z^d)$ a.e. $\w\in\U.$
\item[$(ii)$] $\|\tau f(\w)\|_{\ell^2}$ is a measurable function of $\w$ and
\[
\|f\|^2= \|\widehat{f}\|^2= \int_{\U} \|\tau f(\w)\|_{\ell^2}^2 \, d\w.
\]
\end{enumerate}
\end{lemma}

\begin{lemma}\label{Lema 3.2}
Let $V$ be a finitely generated SIS in $L^2(\R^d).$ Then we have
\begin{enumerate}
\item[$(i)$] $J_V(\w)$ is a closed subspace of $\ell^2(\Z^d)$ for a.e. $\w\in \U.$
\item[$(ii)$] $V=\{f\in L^2(\R^d)\colon \tau f(\w)\in J_V(\w) \mbox{ for a.e. } \w\in \U\}.$
\item[$(iii)$] For each $f\in L^2(\R^d)$ we have that $\|\tau(P_V f)(\w)\|_{\ell^2}$ is a measurable function of the variable $\w$ and
\[
\tau(P_V f)(\w)=  P_{J_V (\w)}(\tau f(\w)).
\]
\item[$(iv)$] Let $\varphi_1, \dots, \varphi_m\in L^2(\R^d).$ We have that
\begin{enumerate}
\item[$(a)$] $\{\varphi_1, \dots, \varphi_m\}$ is a set of generators of $V,$ if and only if $\{\tau\varphi_1(\w), \dots,\tau\varphi_m(\w)\}$
spans $J_{V}(\w)$ for a.e. $\w\in \U.$
\item[$(b)$] The integer translations of $\varphi_1, \dots, \varphi_m$ are a frame (resp. Riesz basis) of $V,$ if and only if
$\tau\varphi_1(\w), \dots,\tau\varphi_m(\w)$ are a frame (resp. Riesz basis) of $J_V(\w)$ with the same frame (resp. Riesz) bounds, for a.e. $\w\in \U.$
\end{enumerate}
\end{enumerate}
\end{lemma}

%We will need
%some definitions and known results concerning extra-invariance for shift invariant spaces. These are described in the next subsection.
%%%%%%%%%%%%%%%%%%%%%%%%%%%%%%%%%%%%%%%%%%%%%%%%%%%%%%%%%%%%%

\subsection{Optimality for the class of SIS in $L^2(\R^d)$}

\medskip
In \cite{ACHM07} the authors give a solution for the case where the approximation class is the class of SIS in $L^2(\R^d).$ For this, they reduce the optimization
problem into an uncountable set of finite dimensional problems in the Hilbert space $\mathcal{H}=\ell^2(\Z^d).$ 
\begin{theorem}[Theorem 2.3 of \cite{ACHM07}]\label{Teorema-original}
Let $\mathcal{F}=\{f_1, \cdots, f_m\}$ be a set of functions in $L^2(\R^d).$ Let $\lambda_1(\w)\ge\dots\ge\lambda_m(\w)$ be
the eigenvalues of the Gramian $G_{\mathcal{F}}(\w).$ Then, there exists $V^*\in\V^{\ell}=\{V\colon V \mbox{ is a SIS of length at most } \ell\}$
such that 
\[
\sum_{i=1}^m\|f_i- P_{V^*} f_i\|^2\le \sum_{i=1}^m\|f_i- P_{V} f_i\|^2, \quad \forall \, V\in\V^{\ell}.
\] 
Moreover, we have that
\begin{enumerate}
\item[(1)] The eigenvalues $\lambda_i(\w),$ $1\le i\le m$ are $\Z^d-$periodic, measurable functions in $L^2(\U)$ and 
\[
\mathcal{E}(\F, \ell)= \sum_{i=\ell+1}^m \int_{\U} \lambda_i(\w)\, d\w.
\]  
\item[(2)] Let $\theta_{i} (\w)= \lambda_i^{-1/2}(\w)$ if $ \lambda_i(\w)$ is different from zero, and zero otherwise. Then, there
exists a choice of measurable left eigenvectors $Y^1(\w), \cdots, Y^{\ell}(\w)$ with $Y^i= (y^i_1,\cdots, y^i_m)^t,$ $i=1, \cdots, \ell,$
associated with the first $\ell$ largest eigenvalues of $G_{\F}(\w)$ such that the functions defined by
\[
\widehat{\varphi}_i(\w)= \theta_i (\w)\sum_{j=1}^m y^i_j(\w) \widehat{f_j}(\w), \quad i=1, \cdots, \ell, \, \w\in\R^d
\]
are in $L^2(\R^d).$ Furthermore, the corresponding set of functions $\Phi=\{\varphi_1, \cdots, \varphi_{\ell}\}$ is a generator
set for the optimal subspace $V^*$ and the set $\{\varphi_i(\cdot -k), k\in \Z^d, i=1, \cdots, \ell\}$ is a Parseval frame for $V^*.$
\end{enumerate}
\end{theorem}

\subsection{Extra invariance}\label{EI}

We will need some definitions and known results concerning extra-invariance for shift invariant spaces. These are described in this subsection.
\begin{definition}
Let $V\subset L^2(\R^d)$ be a SIS. We define the {\it invariance set} as follows
\[
M:=\{x\in\R^d\colon T_x f\in V, \,\forall f\in V\}.
\]
\end{definition}
In \cite{ACHKM10} (see also \cite{ACP11}), the authors proved that the invariance set of a shift invariance space $V\subset L^2(\R^d)$ is a closed additive subgroup of $\R^d$ that contains $\Z^d.$ For instance, in the case of the line the invariant set of a
shift invariant
space could be  $\Z, \frac{1}{n}\Z $ for some  $n\in\N$  or $\R.$  

\begin{definition} Let $\Phi \subset L^2(\R^d)$. We will say that $V=S(\Phi)$ is {\it $M$ extra-invariant} if $T_m f\in V$ for all $m\in M$ and for all $f\in V.$

If $M=\R^d$ we will say that $V$ has {\it total extra-invariance}.
\end{definition}
One example of a translation invariant space in $\R$ is the Paley-Wiener space of functions that are bandlimited  to $[-1/2, 1/2]$ defined by
\[
PW= \{f\in L^2(\R)\colon \mbox{ supp }(\widehat{f})\subseteq [-1/2,1/2]\}.
\]  
It is easy to prove that for a measurable set $\Omega\subset\R^d,$ the space
\begin{equation}\label{Wiener}
V_{\Omega}:=\{f\in L^2(\R^d)\colon \mbox{supp}(\widehat{f})\subset \Omega\}
\end{equation}
is translation invariant. Moreover, Wiener's theorem (see \cite{Hel64}) proves that any closed translation invariant subspace of $L^2(\R^d)$ is of the
form \eqref{Wiener}.

If $V$ is a shift invariant space of length $\ell$ and $M$ is an additive  subgroup of $\R^d$ containing $\Z^d$, we will say that $V$ has {\it extra-invariance} $M$ if $V$ is $M-$invariant.
Note that in this case, if $\Phi$ is a set of generators of $V$, i.e. $V=S(\Phi)$, then 
$$
S(\Phi)= \overline{\mbox{span}}\{T_k\phi\colon \phi\in\Phi, k\in\Z^d\}=
\overline{\mbox{span}}\{T_{\alpha}\phi\colon \phi\in\Phi, \alpha\in M\}.
$$

In \cite{ACHKM10} the authors characterize those shift invariant spaces $V\subset L^2(\R)$ that have extra-invariance. They show that either $V$ is translation invariant, or there exists a maximum positive integer $n$ 
such that $V$ is $\frac{1}{n}\Z-$invariant. 

The d-dimensional case is consider in  \cite{ACP11}. There, a characterization of the extra invariance of $V$ when $M$ is not all $\R^d$ is obtained. 
Given $M$ a closed subgroup of $\R^d$ containing $\Z^d$ and 
$M^*=\{x\in\R^d\colon \langle x,m\rangle\in\Z\quad\forall m\in M\},$ the authors construct a special partition 
$\{B_{\sigma}\}_{\sigma\in\mathcal{N}}$ of $\R^d,$
where each $B_{\sigma}$ is an $M^*-$periodic set and the index set $\mathcal{N}$ is a section of the quotient $\Z^d/M^*.$
More precisely, for each $\sigma\in\mathcal{N},$
\begin{equation}\label{def-Bsigma}
B_{\sigma}=\Omega+\sigma+M^*=\bigcup_{m^*\in M^*} (\Omega+\sigma)+m^*,
\end{equation} 
where $\Omega$ is a section of the quotient $\R^d/\Z^d.$ We refer to \cite{ACP11} for more details.

Using this partition,
for each $\sigma\in\mathcal{N},$ they define the subspaces associated to a given SIS $V$
\begin{equation}\label{Usigma}
V_{\sigma}= \{f\in L^2(\R^d)\colon \widehat{f}= \chi_{B_{\sigma}}\widehat{g}, \text{ with } g\in V\}.
\end{equation}

Given $f\in L^2(\R^d)$ define for $\si \in \mathcal{N}$, the function $f^{\si}$ by $\widehat{f^{\si}}=\widehat{f}\chi_{B_{\sigma}}$.

The authors give a characterization of the $M-$invariance of $V$ in terms of the subspaces 
$V_{\sigma}.$ More specifically they prove that
\begin{theorem}\label{extra-inv}
If $V\subset L^2(\R^d)$ is a SIS and $M$ is a closed subgroup of $\R^d$ containing $\Z^d,$ then the following are equivalent.
\begin{enumerate}
\item[(i)] $V$ is $M-$invariant,
\item[(ii)] $V_{\sigma}\subset V$ for all $\sigma\in\mathcal{N},$
\item[(iii)] $J_{V_{\sigma}}(\w)\subset J_{V}(\w)$ for almost every $\w$ and each $\sigma\in\mathcal{N},$ 
\item[(iv)] if $V=S(\Phi)$  then $\tau\varphi^{\si}(\w) \in J_V(\w)$ a.e. $\w\in\U$
for all  $\varphi\in\Phi$ and all $\si\in\mathcal{N}$.
\end{enumerate}
\end{theorem}

\section{Optimality for the class of SIS with extra-invariance}\label{section1}

Here we consider the approximation problem for the class of finitely generated SIS with extra invariance under a given  {\it proper} subgroup  $M$ of $\R^d.$ 

Let us start introducing some notation. Let $m, \ell\in\N,$ $M$ be a closed proper subgroup of $\R^d$ containing $\Z^d$ and $\F=\{f_1,\dots, f_m\}\subset L^2(\R^d).$
 Define 
\begin{equation}\label{classV}
\V_M^{\ell}=\{ V: V \;{\text {is a SIS of  length at most}}\;\ell \;{\text {and}} \;  V {\text{ is }} M{\text {-invariant}}\}.
\end{equation}

Let $\mathcal{N}=\{\si_1,\dots,\si_{\kappa}\}$ be a section of the quotient $\Z^d/M^*$ and $\{B_{\si}\::\si\in \mathcal{N}\}$ the partition defined in \eqref{def-Bsigma}.
%considered in Subsection \ref{EI}. {\color{red} citar def}

For each ${\sigma}\in\mathcal{N},$ 
we consider $\F^{\sigma}=\{f_1^{\sigma},\dots, f_m^{\sigma}\}\subset L^2(\R^d)$ where, $f_j^{\sigma}$ is such that
$\widehat{f_j^{\sigma}}= \widehat{f_j}{\chi_{B_{\sigma}}}$ for $j=1, \dots, m.$ 
Also, let \mbox{$\widetilde{\F}=\{f_1^{\sigma_1},\dots, f_m^{\sigma_1},\dots\dots,f_1^{\sigma_{\kappa}},\dots, f_m^{\sigma_{\kappa}}\}.$}

For each $\w\in\U$ let $G_{\widetilde{\F}}(\w)$ be the associated Gramian matrix of the vectors in $\widetilde{\F}$ with eigenvalues 
\[
\lambda_1(\w)\ge\cdots\ge \lambda_{m\kappa}(\w)\ge0.
\]
Using Lemma \ref{medibilidad}, we have that these eigenvalues are measurable functions. 

Since  $f_i^{\si_s}$ is orthogonal to $f_i^{\si_t}$ if $s \neq t$, the Gramian $G_{\widetilde{\F}}(\w)$
is a diagonal  block matrix with  blocks $G_{\sigma}(\w), \;\sigma\in\mathcal{N}.$ Here $G_{\sigma}(\w)$ is the $m\times m$ Gramian associated to the data $\F^{\si}.$
On the other hand, using Lemma \ref{medibilidad} we have that 
\[
G_{\sigma}(\w)= U_{\sigma}(\w) \Lambda_{\si}(\w) U_{\sigma}^*(\w) \quad a.e. \;\;\w \in \U
\]
where 
 $U_{\si}$ are unitary   and $\Lambda_{\sigma}(\w):=\text{diag}(\lambda_1^{\sigma}(\w),\dots, \lambda_m^{\sigma}(\w))\in\C^{m\times m}$ 
 and they are also measurable matrices as in Lemma \ref{medibilidad}. We also have $\lambda_1^{\sigma}(\w)\ge \dots \ge \lambda_m^{\sigma}(\w)$ for each $\sigma\in\mathcal{N}.$
 
 Using the decompositions of the blocks $G_{\sigma}$ we have that 
 \begin{equation}\label{P1}
  G_{\widetilde{\F}}(\w) = U(\w) \Lambda(\w) U^*(\w)
 \end{equation}
 where $U$ has blocks $U_{\si}$ in the diagonal, and $\Lambda$ is diagonal with blocks $\Lambda_{\si}$.
We want to recall here that for almost each $\w$ the matrix $\Lambda(\w)$ collects all the eigenvalues of the Gramian $G_{\widetilde{\F}}(\w)$ and
the columns of the matrix $U(\w)$ are the associated left eigenvectors. Note that an  eigenvector associated to the
eigenvalue $\la_j^{\si}(\w)$ has all the components not corresponding to the block $\si$ equal to zero.

Now for each fixed $\w\in\U,$ we consider $\{(i_1(\w),j_1(\w)), \dots, (i_{n}(\w),j_{n}(\w))\}$ with $i_s(\w)\in\mathcal{N}$ and $j_s(\w)\in\{1,\dots, m\}$ and $n=m\kappa$ such that 
\[
\lambda_{j_1(\w)}^{i_1(\w)}\ge \dots \ge \lambda_{j_{n}(\w)}^{i_{n}(\w)}\ge 0
\]
are the ordered eigenvalues of $G_{\widetilde{\F}}(\w),$ with corresponding left eigenvectors  $Y^{(i_s(\w), j_s(\w))}\in\C^n,$ for $s=1, \cdots, n.$
 
Here $i_s(\w)$ indicates the block of the matrix $G_{\widetilde{\F}}(\w)$ in which the eigenvalue $\lambda_{j_s(\w)}^{i_s(\w)}(\w)$ is found 
and $j_s(\w)$ indicates the displacement in this block of the matrix $G_{\widetilde{\F}}(\w)$. More precisely, we have that  $\lambda_{j_s(\w)}^{i_s(\w)}(\w)$ coincides with $\lambda_{(i_s(\w)-1) m +j_s(\w)}(\w),$ the $((i_s(\w)-1) m +j_s(\w))-$th eigenvalue of $G_{\widetilde{\F}}(\w).$
When $\w\in \U$ is fixed, we will write $i_s$ instead of $i_s(\w)$ and  $j_s$ instead of $j_s(\w).$ 

We will prove now that  $\gamma_s(\w):=\lambda_{j_{s(\w)}}^{i_s(\w)}(\w)$ is measurable as a function on $\w$ for each $s=1, \cdots, n.$

Let $s\in\{1, \dots, n\}$ fixed. Let $i_s(\w)\in\mathcal{N}$ and $j_s(\w)\in\{1,\dots, m\}.$ We have that $\gamma_s(\w)= \lambda_j^{\sigma}(\w)$ for all $\w\in E_{{\sigma}j}:= \{\w\in\U\colon i_s(\w)= \sigma,  j_s(\w)=j \}.$

We observe that 
\[
E_{\sigma j}= \{\w\in\U\colon i_s(\w)=\sigma,  j_s(\w)=j\}= \{\w\in \U\colon \lambda_s(\w)= \lambda_j^{\sigma}(\w)\}.
\]
Using Lemma \ref{medibilidad} applied to $G_{\widetilde{\F}}(\w)$ and  $G_{\sigma}(\w),$ we have that $\lambda_s$ and  $\lambda_j^{\sigma}$ are 
measurable functions of $\w.$ Therefore $E_{\sigma j}$ are measurable sets.

We further observe that $\gamma_s(\w)=  \lambda_j^{\sigma} (\w),$ for $\w\in E_{\sigma j}.$ So $\gamma_s(\w)$ is a measurable function.
A similar argument shows that the eigenvectors are measurable.

Finally we define $h_s\colon \R^d\to \C,$ for $s=1,\dots, \ell$
\begin{equation}\label{def-h}
h_s(\w):= \theta_{j_s}^{i_s}(\w) \sum_{k=1}^m y_{(i_s-1)m+k}^{(i_s, j_s)}(\w) \widehat{f}_k^{i_s}(\w),
\end{equation}
where $\theta_{j_s}^{i_s}(\w)= (\lambda_{j_s}^{i_s}(\w))^{-1/2}$ if $\lambda_{j_s}^{i_s}(\w)\neq 0$ and $\theta_{j_s}^{i_s}(\w)=0$ otherwise.

Now we are ready to state the main result of this section.
%%%%%%%%%%%%%%%%%%%%%%%%%%%%

\begin{theorem}\label{solucion-PB}
Let $m, \ell\in\N,$ and $M$ be a closed proper subgroup of $\R^d$ containing $\Z^d.$ Assume that $\F=\{f_1,\dots, f_m\}\subset L^2(\R^d)$ is given data  and let $\V_M^{\ell}$ be the class defined in \eqref{classV}. Then, there exists a shift invariant space $V^*\in\V_M^{\ell}$ such that
\begin{equation}\label{solution}
V^*= \mathop{argmin}_{V\in\V_M^{\ell}} \sum_{j=1}^{m} \|f_j-P_{V} f_j\|^2. 
\end{equation}

Furthermore, with the above notation,
\begin{enumerate}
\item[(1)] The eigenvalues $\{\lambda_j^{\si}(\w): \si \in \mathcal{N} , j=1,\dots ,m\},$  are $\Z^d-$periodic, measurable functions in $L^2(\U)$ and the error of  approximation is
\[
\mathcal{E}(\F, M,  \ell) :=  \sum_{j=1}^{m} \|f_j-P_{V^*} f_j\|^2=  \int_{\U}\sum_{s= \ell +1}^{m\kappa} \lambda_{j_s}^{i_s}(\w) \;d\w.
\]  
\item[(2)] The functions $\{h_1,\dots, h_{\ell}\}$ defined in \eqref{def-h} are in  $L^2(\R^d)$ and if $\varphi_1,\dots,\varphi_{\ell}$
are defined by $\widehat{\varphi_j} = h_j$, then $\Phi=\{\varphi_1, \cdots, \varphi_{\ell}\}$ is a generator
set for the optimal subspace $V^*$ and the set $\{\varphi_i(\cdot -k), k\in \Z^d, i=1, \cdots, \ell\}$ is a Parseval frame for $V^*.$
\end{enumerate}
\end{theorem}

\color{black}

\begin{proof}

Let $\V^{\ell}$ be the class defined in Theorem \ref{Teorema-original}, that  is $\V^{\ell}$ is the set of all shift invariant spaces $V$ that can be generated by $\ell$ or less generators. (Note that we do not ask the elements of the class $\V^{\ell}$ to have extra invariance.)

Define $V^* \in \V^{\ell}$ to be the optimal space given by Theorem \ref{Teorema-original} for the data $\widetilde{\F}$.
That is,
\begin{equation}\label{eq-optimal}
\sum_{\si \in \mathcal{N}} \sum_{j=1}^m \|f_j^{\sigma} - P_{V^*}f_j^{\sigma}\|^2 \le \sum_{\si \in \mathcal{N}} \sum_{j=1}^m \|f_j^{\sigma} - P_{V}f_j^{\sigma}\|^2 \qquad \forall \;\;V \in \V^{\ell}.
\end{equation}

We claim that $V^* \in \V_M^{\ell}$ (in particular is $M$ extra-invariant) and it is optimal in this class for the data $\F,$ i.e.
\begin{equation}\label{optimality}
 \sum_{j=1}^m \|f_j - P_{V^*}f_j\|^2 \le \sum_{j=1}^m \|f_j - P_{V}f_j\|^2 \qquad \forall \;\;V \in \V_M^{\ell}.
\end{equation}

Let us prove first that $V^*$ is $M$ extra-invariant. For this we will  check that the generators of $V^*$ satisfy 
condition (iv) in Theorem \ref{extra-inv}. 
We have from   \eqref{P1} that the Gramian $G_{\widetilde{\F}}(\w)$ can be decomposed as $G_{\widetilde{\F}}(\w) = U(\w)\Lambda(\w)U^*(\w)$ with 
	eigenvalues $\{\la_j^{\si}(\w) : \si\in\mathcal{N}, j=1,\dots,m\}.$

By Theorem \ref{Teorema-original}, the $\ell$ generators of $V^*$ have the form defined in \eqref{def-h},
\begin{equation}\label{P2}
\widehat{\varphi}_{s}(\w)=\theta_{j_s}^{i_s}(\w) \sum_{k=1}^m y_{(i_s-1)m+k}^{(i_s, j_s)}(\w) \widehat{f}_k^{i_s}(\w), \qquad {\text{ for }}
s=1,\dots,\ell.
\end{equation}

From \eqref{P2} it is clear that $\widehat{\varphi}_{s}$ is supported in $B_{i_s}$, since each $\widehat{f}_k^{i_s}$ is supported in $B_{i_s}$.
Then if we apply the cut off operator to these generators we obtain $\widehat{\varphi}_{s}^{\si}(\w) = \widehat{\varphi}_{s}(\w) {\text{ if }} \si = i_s(\w) {\text{ and }} \widehat{\varphi}_{s}^{\si}  = 0 \;{\text {otherwise}}.$
So, in any case $\varphi_{s}^{\si} \in V^*$ for all $\si \in \mathcal{N},\;\; s=1,\dots, \ell$ which proves the $M$-invariance of $V^*.$

What is left now is to prove that $V^*$ is optimal over the class $\V_M^{\ell}$, that is $V^*$ satisfies equation \eqref{optimality}.
For this note that if $V\in \V_M^{\ell}$ then $V = \bigoplus_{\si \in \mathcal{N}} V_{\si}.$ So we have for any $f\in L^2(\R^d)$,
\begin{equation*}
\|P_{V} f\|^2 = \|P_{V} \sum_{\si \in \mathcal{N}} f^{\si}\|^2 = \| \sum_{\si \in \mathcal{N}} P_{V}f^{\si}\|^2 = 
\|\sum_{\si \in \mathcal{N}} P_{V^{\sigma}} f^{\sigma}\|^2=
\sum_{\si \in \mathcal{N}}\|P_{V^{\sigma}} f^{\si}\|^2=
\sum_{\si \in \mathcal{N}}\|P_{V} f^{\si}\|^2,
\end{equation*}
which implies together with \eqref{eq-optimal} that 

\begin{equation*}
 \sum_{j=1}^m \| P_{V^*}f_j\|^2 \ge \sum_{j=1}^m \| P_{V}f_j\|^2, \qquad \forall \;\;V \in \V_M^{\ell}.
\end{equation*}
The others claims of the theorem are a direct consequence of Theorem  \ref{Teorema-original}.

\end{proof}
%%%%%%%%%%%%%%%%%%%%%%%%%%%%%%%%%%
%%%%%%%%%%%%%%%%%%%%%%%%%%%%%%%%%
\bigskip
\section{Approximation with Paley-Wiener spaces}\label{section2}

\subsection{Preliminaries}

In this section the class of approximation subspaces  will be finitely  generated SIS with total translation invariance.
That is translation invariant spaces that are generated by the integer translates of a finite number of functions.

 More precisely, given $\ell\in\N$ define $\T^{\ell}$ to be the set of all shift invariant spaces $V=S(\varphi_1,\dots, \varphi_{\ell})$  for some functions $\varphi_1,\dots, \varphi_{\ell}$ in $L^2(\R^d)$, and such that $V$ is  translation invariant and
 the integer translates of $\{\varphi_1,\dots, \varphi_{\ell}\}$ form a Riesz basis of $V$.
 
 Given a set $\F=\{f_1,\dots, f_m\}\subset L^2(\R^d),$ we want to find $V^{*}\in \T^{\ell}$ such that
 \begin{equation}\label{solucion-Tl}
V^*= \mathop{argmin}_{V\in\T^{\ell}} \sum_{j=1}^{m} \|f_j-P_{V} f_j\|^2. 
\end{equation}
Here  $P_{V}$ denotes the orthogonal projection on $V.$

Before going to the approximation problem, we will obtain a characterization of the class $\T^{\ell}.$

Using Wiener's theorem, we have that $V$ is a translation invariant space in $L^2(\R^d)$ if and only if there exists a measurable set $\Omega
\subset\R^d$ such that 
$$V= \{f \in L^2(\R^d): \hat{f}(\omega)=0 {\text{  a.e.  }} \omega \in \R^d \setminus \Omega  \}.$$ 
Since $\Omega$ is unique up to measure zero, we will write $V = V_{\Omega}.$ 

\begin{definition}\label{multitile}
Let $\Omega\subset\R^d$ be measurable and $L\subset \R^d$ be a countable set. We say that $\Omega$ {\it tiles $\R^d$ when translated by $L$ at 
level $\ell\in\N$}
if
\[
\sum_{t\in L} \chi_{\Omega}(\w-t)=\ell, \quad \mbox{ for a.e. }\w\in\R^d.
\]
In case of $L=\Z^d$ we will say that $\Omega$ is an {\it $\ell$ multi-tile}.
\end{definition}

It is known (see for example \cite{Ko13},)  that $\Omega$ is an  $\ell$ multi-tile of $\R^d$, if and only if, up to measure zero, $\Omega$ is the union of $\ell$ measurable and disjoint
$1$ tile sets. i.e. $\Omega$ is a quasi-disjoint union of $\ell$ sets of representatives of $\R^d/\Z^d.$

\begin{lemma}\label{multi}

A measurable set $\Omega\subset \R^d, \;\ell$ multi-tiles $\R^d$ 
if and only if  $$\Omega = \Omega_1\cup\dots\cup \Omega_{\ell } \cup N,$$ where $N$ is a zero measure set, 
and the sets $\Omega_j$,  $ 1\leq j\leq \ell$ are measurable, disjoint and each of them tiles $\R^d$ by translations on $\Z^d$.
\end{lemma}

The following  proposition characterizes the set $\Omega$ for   the elements in  $\T^{\ell}.$
\begin{proposition}\label{propPW}
A subspace
$V$ is in $\T^{\ell}$ if and only if 
 $V=V_{\Omega}$ with $\Omega$  a measurable $\ell$ multi-tile of $\R^d.$ 
\end{proposition}

\begin{proof}

Assume first that  $V \in \T^{\ell}$, so $V=V_{\Omega}$ for some measurable $\Omega \subset \R^d$.
Also, as a consequence of Wiener's theorem,  for almost all $\w \in \U$ we have $J_V(\w) \cong \ell^2(O_{\w})$  with 
$O_{\w}=\{ k\in \Z^d: w+k \in \Omega\}.$ To see this, we note that $J_V(\w) \subset \ell^2(O_{\w})$. For the other inclusion,
fix $\w\in \U.$ Using that $\Omega = \bigcup_{k\in\Z^d}E_k$ where $E_k = (\U+k) \cap \Omega$, we have that  $k\in O_{\w}$,
if and only if  $\w + k \in E_k$. Hence,
 if $ a \in \ell^2(O_{\w})$  consider the function  $G_{\w} (\xi) = \sum_{k \in O_{\w}} a_k \chi_{E_k}(\xi)$.
 Since $G_{\w}$ is in $L^2(\Omega),$
 the function $g$ defined by $\widehat{g}=G_{\w}$ is in $V$,  and $\widehat{g}(\w+k) = a_k$
 if $k \in O_{\w}.$ Therefore, $g \in V$ and $a = \tau g(\w) \in  J_V(\w).$  

Now, since $V=S(\varphi_1,\dots, \varphi_{\ell}),$  and the integer translates of $\varphi_1,\dots, \varphi_{\ell}$ form a Riesz basis
of $V,$ using Lemma \ref{Lema 3.2} we obtain that $\{\tau\varphi_1(\w), \dots,\tau\varphi_{\ell}(\w)\}$ form a Riesz basis of $J_V(\w)$ with the same Riesz bounds for a.e. $\w\in \U.$
We conclude that dim($J_V(\w))= \ell$ a. e. $\w \in \U.$

Since $V$ is translation invariant, by the observation above dim$(J_V(\w))= \#O_w$. Then $\#O_{\w} = \ell$ for almost all $\w \in \U,$
which implies that $\Omega$ is an $\ell$ multi-tile. (Here $\#A$ denote the cardinal of the set $A).$

For the converse, assume that $\Omega$ is a measurable $\ell$ multi-tile of $\R^d$. Define $V = V_{\Omega}$.
So, $V$ is translation invariant.

By Lemma \ref{multi} we have that $\Omega = \Omega_1\cup\dots\cup \Omega_{\ell }$  up to a measure zero set,
where each $\Omega_j$ is a set of representatives or $\R^d/\Z^d.$
We define $\varphi_j$ by its Fourier transform:  $\widehat\varphi_j = \chi_{\Omega_j},\quad j=1,\dots,\ell.$

Since $\{  e^{2\pi i \w k} \widehat\varphi_j: k \in \Z^d\}$ is an orthonormal basis of $L^2(\Omega_j),$ we have that
$\{e^{2\pi i \w k} \widehat\varphi_j : k \in \Z^d, j=1,\dots, \ell \}$ is an orthonormal basis of  $L^2(\Omega),$ and so,
$\{t_k\varphi_j:k \in \Z^d, j=1,\dots, \ell \}$ is an orthonormal basis of $V$, in particular a Riesz basis.
\end{proof}

\subsection{The approximation problem for Paley-Wiener Spaces}

Now we come back to our approximation problem.
In order to find an optimal subspace in the class $\T^{\ell}$ for a set of data $\F=\{f_1,\dots, f_m\},$
it is enough to find the associated $\ell$ multi-tile $\Omega$ in $\R^d.$
It is not difficult to see that if we allow $\Omega$ to be {\it {any}} $\ell$ multi-tile the minimum in \eqref{solucion-Tl} may not exist.
So  we will restrict $\Omega$ to be inside a cube that could be arbitrarily large.
Let us fix $N\in\N$.
Define  
\begin{align*}
C_N &:= [-(N+1/2),N+1/2]^d,\\
M_N^{\ell} &:=\{\Omega\subset C_N: \, \Omega \mbox{ is measurable and $\ell$ multi-tiles }  \R^d\} {\text{ and }}\\
\T_N^{\ell} &:= \{V\in \T^{\ell}: V= V_{\Omega} {\text{ with }} \Omega \in M_N^{\ell} \}.
\end{align*}

With this notation we can state the main result of this section.

\begin{theorem}
Assume that  $m, \ell \in\N$  and a set $\F=\{f_1,\dots, f_m\}\subset L^2(\R^d)$, are given. Then for each $N \geq \ell$
 there exists a Paley-Wiener space $V^*\in \T^{\ell}_N$ that satisfies
  \begin{equation}
V^*= \mathop{argmin}_{V\in\\ \T^{\ell}_N} \sum_{j=1}^{m} \|f_j-P_{V} f_j\|^2, 
\end{equation}
where $\T^{\ell}_N$ is the class defined above.
\end{theorem}

\begin{proof}

First we observe that if a solution space $V^*$ exists then

\begin{equation}\label{Problem2-particular}
V^*= \mathop{argmin}_{V\in\T^{\ell}_N} \sum_{j=1}^{m} \|f_j-P_{V} f_j\|^2 \;
= \mathop{argmax}_{V\in\T^{\ell}_N} \sum_{j=1}^{m} \|P_{V} f_j\|^2,
\end{equation}
and using the definition of $\T^{\ell}_N,$ we have that 
\begin{equation}\label{maximun}
\mathop{max}_{V\in\T^{\ell}_N} \sum_{j=1}^{m} \|P_{V} f_j\|^2= \mathop{max}_{\Omega\in
M^{\ell}_N} \sum_{j=1}
^{m} \|P_{V_{\Omega}} f_j\|^2.
\end{equation}

So, we need to find $\Omega \in M^{\ell}_N$ that yields the maximum in \eqref{maximun}.

Using Lemma \ref{Lema 3.1} we see that for each $\Omega \in M^{\ell}_N$,
\begin{align}\label{xx}
\sum_{j=1}^{m} \|P_{V_{\Omega}} f_j\|^2&= \sum_{j=1}^{m} \|P_{\widehat{V_{\Omega}}} \widehat{f_j}\|^2\\
\nonumber&= \sum_{j=1}^{m} \int_{\U}  \|P_{J_{V_{\Omega}}}(\w)( \tau f_j(\w))\|^2_{\ell^2(\Z^d)}\, d\w\\
\nonumber&= \int_{\U} \sum_{j=1}^{m}  \|P_{J_{V_{\Omega}}}(\w) (\tau f_j(\w))\|^2_{\ell^2(\Z^d)}\, d\w.
\end{align}

Recall that $P_{J_{V_{\Omega}}}(\w) $ denotes the orthogonal projection onto the closed subspace ${J_{V_{\Omega}}}(\w)$ of $\ell^2(\Z^d).$

Furthermore, if $\Omega \in M^{\ell}_N,$ we know from the proof of  Proposition \ref{propPW} that $\dim(J_{V_{\Omega}}(\w))=
\ell$  for a.e. $\w\in \U .$ Note that $J_{V_{\Omega}}(\w)$ agrees with the subspace of $\ell^2(\Z^d)$ of the sequences supported 
in $O_{\w}.$
Then there exists a unique set of $\ell$ integer vectors ${\bf{k}}^{\Omega}(\omega)=\{k_1^{\Omega}(\w),\dots, k_{\ell}^{\Omega}(\w)\}\subset \Z^d$ such that  $\mbox{span}\{\delta_{k_j^{\Omega}(\w)}: j=1,\dots,\ell\}= J_{V_{\Omega}}(\w),$
for a.e. $\w\in \U$. Here $\delta_j$ denotes the canonical vector in $\ell^2(\Z^d).$
i.e. $\delta_j(s) = 0$ if $s \neq j$ and $1$ otherwise.
Note that , since $\Omega \subset C_N$ necessarely $\|k_j^{\Omega}(\w)\|_{\infty} \leq N,$ for each $j$ and $\omega.$
Combining this observation with \eqref{xx} we obtain,
\begin{equation}\label{xxx}
\sum_{j=1}^{m} \|P_{V_{\Omega}} f_j\|^2=  \int_{\U} \sum_{j=1}^{m} \sum_{s=1}^{\ell}  |\widehat{f_j}(\w+k_s^{\Omega}(\w))|^2\, d\w.
\end{equation}

So, now we need to maximize the left hand side in \eqref{xxx} over all the sets $\Omega \in M^{\ell}_N.$

Note that given  $\Omega \in M^{\ell}_N,$ for almost each $\omega \in \U,$ the set $\Omega$ contains exactly $\ell$ elements
from the sequence $\{\omega + k, k \in \Z^d\}.$
Then we can pick for each $\omega \in \U$ (up to a set of zero measure)  $\ell$ translations $k_s^*(\omega)$ 
such that  $\sum_{j=1}^{m}  \sum_{s=1}^{\ell}  |\widehat{f_j}(\w+k_s^*(\w))|^2$ is maximum over all
sets of $\ell$ translations ${\bf{k}}=\{k_1,\dots, k_{\ell}\}\subset \Z^d,$ with $\|k_j\|_{\infty} \leq N.$ 
The maximum exists since the fibers of $f_j$ are $\ell^2(\Z^d)$-sequences and the number of translations considered 
is finite.

Call $\mathcal{K}$ the set of admisibles translations i.e. $\mathcal{K} = \{{\bf{k}}=\{k_1,\dots, k_{\ell}\}\subset \Z^d :\|k_j\|_{\infty} \leq N\}$ and for   ${\bf{k}} \in \mathcal{K}$ set $ H_{\bf{k}} (\w) = \sum_{j=1}^{m}  \sum_{s=1}^{\ell}  |\widehat{f_j}(\w+k_s(\w))|^2.$ 

Our goal is to construct a set  $\Omega$ such that the associated space $V_{\Omega}$ is optimal. So the idea is to construct the optimal set $\Omega^*$
considering for each $\omega \in \U$ the optimal translations $\{\omega + k_s^*(\w) : s=1,\dots,\ell\},$ and then taking the union over almost all $\w \in \U.$

For this we define for each ${\bf{k}} =\{k_1,\dots, k_{\ell}\} \in \mathcal{K}$ the following subset of $\U,$
\[
E_{\bf{k}}=\left\{\w\in \U\colon   H_{\bf{k}} (\w)\geq  H_{\bf{r}} (\w),\;
 \forall \;{\bf{r}}=\{{r_1},\dots,{r_{\ell}}\} \in \mathcal{K}\right\},
\]
i.e., $E_{\bf{k}}$ is the set of $\w\in \U$ for which the maximum is attained for ${\bf{k}}=\{k_1,\dots, k_{\ell}\}.$
Note that $E_{\bf{k}}$ could be the empty set for some ${\bf{k}}=\{k_1,\dots, k_{\ell}\}$ and the sets
$E_{\bf{k}}$ may not be disjoint.

Finally we define our optimal set as,
\[
\displaystyle \Omega^*= \bigcup_{\begin{smallmatrix}
  {\bf{k}} \in \mathcal{K}
 \end{smallmatrix}} \bigcup_{j=1}^{\ell} E_{\bf{k}} +k_j.
\]

We will now prove  that $\Omega^*$ is measurable.  First we note that $E_{\bf{k}}$ is a measurable set for each ${\bf{k}} \in \mathcal{K}$ since,
\[
\displaystyle E_{\bf{k}}= \bigcap_{{\bf{r}}\in \mathcal{K}} F_{{\bf{r}}}^{\bf{k}},
\]
where,
\[
F_{{\bf{r}}}^{\bf{k}}=\left\{\w\in \U\colon   H_{\bf{k}} (\w)\geq  H_{\bf{r}} (\w)
\right\}.
\]
Now, since $F_{\bf{r}}^{\bf{k}}$ is  measurable  for all ${\bf{r}}\in \mathcal{K},$ we obtain that $E_{\bf{k}}$ is  measurable and so is  $\Omega^*$.

Furthermore, by construction, $\Omega^*$ is in $M_N^{\ell}.$ Since for all $\Omega \in M_N^{\ell}$ we have that,
$$
 \sum_{j=1}^{m} \sum_{s=1}^{\ell}  |\widehat{f_j}(\w+{k_s^{\Omega}}(\w))|^2
\leq \sum_{j=1}^{m} \sum_{s=1}^{\ell}  |\widehat{f_j}(\w+k_s^*(\w))|^2 \;\;{\text{for almost all }} \w \in \U,
$$
 taking the integral over $\U$ we get
\begin{align*}
\sum_{j=1}^{m} \|P_{V_{\Omega}} f_j\|^2&=  \int_{\U} \sum_{j=1}^{m} \sum_{s=1}^{\ell}  |\widehat{f_j}(\w+k_s^{\Omega}(\w))|^2\, d\w \\ 
&\leq \int_{\U} \sum_{j=1}^{m} \sum_{s=1}^{\ell}  |\widehat{f_j}(\w+k_s^*(\w))|^2\, d\w=\sum_{j=1}^{m} \|P_{V_{\Omega^*}} f_j\|^2 .
\end{align*}

This shows that $\Omega^*\in M_N^{\ell}$  is optimal  over all $\Omega\in M_N^{\ell}.$ 
We conclude 
that $V_{\Omega^*}\in\T^{\ell}_N$ is a solution for the data $\F$.  
\end{proof}

\begin{remark}
Notice that if $\Omega_N^*$ is the optimal multi-tile set for the class $\T^{\ell}_N$ for some data $\mathcal{F}$,
then the approximation error is given by 
\[
\mathcal{E}_N(\F,\ell) = \int_{\R^d \setminus \Omega_N^*}   \sum_{j=1}^{m}  |\widehat{f_j}(\w)|^2 d\w =
 \int_{C_N\setminus \Omega_N^*}   \sum_{j=1}^{m}  |\widehat{f_j}(\w)|^2 d\w + 
  \int_{\R^d\setminus C_N}   \sum_{j=1}^{m}  |\widehat{f_j}(\w)|^2 d\w .
\]
Clearly $\mathcal{E}_N(\F,\ell)\geq \mathcal{E}_{N+1}(\F,\ell)$. 
So  $\mathcal{E}(\F,\ell):=\lim_{N\rightarrow \infty} \mathcal{E}_N(\F,\ell)$ is somehow the optimal error.
Since $\F\subset L^2(\R^d)$ then the second integral goes to zero when $N$  goes to infinite,
for functions with good decay at infinite we will be close to the optimal error for conveniently large $N.$
\end{remark}
\begin{remark}
In Proposition \ref{propPW} we show, for an element of $\T^{\ell}$, how to construct a set of generators that gives a Riesz basis of translates in $\Z^d.$  There are many ways to construct other sets of generators that gives Riesz basis of translates. 
Recently Grepstad-Lev in \cite{GL14} constructed a basis of exponentials  for $L^2(\Omega)$ when $\Omega \subset \R^d$ is  a multi-tile.
Later on, Kolountzakis \cite{Ko13} gave a simpler proof of this result in a slightly more general form.
Precisely they prove the following result.
\begin{theorem}[Theorem 1,\cite{Ko13}]\label{exponenciales}
Suppose $\Omega\subset\R^d$ is bounded, measurable and multi-tiles $\R^d$ when translated by $\Z^d$ at level $\ell.$  Then there exist vectors 
$a_1,\dots, a_{\ell}\in\R^d$ such that the exponentials
\[
e^{-2\pi i(a_j+k)\w} \quad j=1, \dots, \ell, \quad k\in\Z^d
\]
form a Riesz basis for $L^2(\Omega).$
\end{theorem}

From  Theorem \ref{exponenciales}, we can obtain immediately  a set of generators for $V_{\Omega}$.
Let $\varphi$ be such that $\widehat{\varphi} = \chi_\Omega.$
If $a_1,\dots, a_{\ell}\in\R^d$ are as in Theorem \ref{exponenciales}, then $V_{\Omega}= S(\varphi_1,\dots,\varphi_{\ell})$ with 
$\varphi_j = t_{a_j}\varphi,\;j=1,\dots,\ell,$ and the translates of $\varphi_1,\dots,\varphi_{\ell}$ form a Riesz basis of $V_{\Omega}.$

In general, all the Riesz basis for $V_{\Omega}$ can be described in the following way:

Let $A = \{a_{js}\}\in [L^2(\U)]^{\ell\times\ell}$ be a measurable matrix, such that $0 < c_1 \leq \lambda(\w) \leq c_2$ for every eigenvalue $\lambda(\w)$
and for almost each $\w \in \U$. Set  $k(\w) = (k_1(\w),\dots, k_{\ell}(\w))$ such that $w+k_s(\w) \in \Omega.$ 
Define $\varphi_j$ such that $\widehat\varphi_j(\w+k_s(w)) = a_{js}(\w)$. Using  the results stated in subsection \ref{sis} it is not difficult to see that
$\varphi_1,\dots,\varphi_{\ell}$ are measurable, $V_{\Omega}= S(\varphi_1,\dots,\varphi_{\ell})$ and the translates of $\varphi_1,\dots,\varphi_{\ell}$ form a Riesz basis of $V_{\Omega}.$

\end{remark}

%%%%%%%%%%%%%%%%%%%%%%%%%%%%%%%%%%%%%%%%%%%%
%%%%%%%%%%%%  Discrete Case  %%%%%%%%%%%%%%%%%%%%%%%%
%%%%%%%%%%%%%%%%%%%%%%%%%%%%%%%%%%%%%%%%%%%%%

\section{The discrete case}\label{section0}
%
%The approximation problem when we consider the class of the finite dimensional subspaces, can be solved
%using Singular Valued Decomposition techniques. This result is an adaptation of the Eckart-Young theorem (\cite{EY36, Sch07}).
%
%\begin{theorem}[Theorem 4.1 of \cite{ACHM07}]\label{TeoremaSVD}
%Let $\mathcal{H}$ be a Hilbert space, $\F=\{f_1, \dots, f_m\}\subset \mathcal{H},$ and $n \in \N.$ 
%Define $X=\mbox{ span }\{f_1, \dots, f_m\},$ and let 
%$\lambda_1\ge \dots \ge \lambda_m$ be the eigenvalues of the matrix $\mathcal{B}(\F)$ defined as $[\mathcal{B}(\F)]_{i,j}= \langle f_i, f_j\rangle_{\mathcal{H}}$ and 
%$y_1, \dots, y_m\in\C^m,$ with $y_i=(y_{i_1}, \dots, y_{i_m})^t$ orthonormal left eigenvectors associated to the eigenvalues $\lambda_1,\dots, \lambda_m.$
%Let $r=\dim X=\mbox{ rk }\mathcal{B}(\F).$
%Define the vectors $q_1, \dots, q_n\in \mathcal{H}$ by
%\[
%q_i= \theta_i \sum_{j=1}^m y_{i_j} f_j, \quad i=1, \dots, n
%\]
%where $\theta_i= \lambda_i^{-1/2}$ if $\lambda_i\neq 0$ and $\theta_i=0$ otherwise.  
%Then $\{q_1, \dots, q_n\}$ is a Parseval frame of $W^*=\mbox{ span }\{q_1, \dots, q_n\}$ and the subspace $W^*$
%is optimal in the sense that 
%\[
%\sum_{i=1}^m\|f_i- P_{W^*} f_i\|^2\le \sum_{i=1}^m\|f_i- P_{W} f_i\|^2, \quad \forall \, W,\; \dim(W)\le n.
%\] 
%Furthermore we have the following formula for the error
%\[
%\mathcal{E}(\F, n)= \sum_{i=n+1}^m \lambda_i.
%\]
%\end{theorem}
%
%\medskip
%
The optimal subspace $V^*$ in 
Theorem \ref{solucion-PB}  is the closest to the data $\F$
over all subspaces $V$   in the class $\V_M^{\ell}.$ 
It is not difficult to see that almost each fiber space  $J_{V^*}(\w) \subset \ell^2(\Z^d)$ of $V^*$ is the closest
to the fibers of our data, $\tau(\F)(\w)=\{\tau f_1(\w),\dots,\tau f_m(\w)\}$ over a certain class of closed subspaces of 
$\ell^2(\Z^d)$ that we will call $\mathcal{D}_{\mathcal{N}}^{\ell}.$
Clearly this class is determine by the class $\V_M^{\ell}$. Therefore it is interesting to study independently the approximation problem for this discrete case.

Consequently  the Hilbert space we consider   in this section is $\ell^2(\Z^d).$
We define  the class of approximating subspaces in the following way:

Let $\mathcal{N}$ be an arbitrary finite set and  $\{D_{\sigma}\colon \sigma\in\mathcal{N}\}$  a partition of $\Z^d,$ that is, $\Z^d= \bigcup_{\sigma\in\mathcal{N}} D_{\sigma},$ where the union is disjoint. 

For $a \in \ell^2(\Z^d),$ we denote $a^{\sigma} = {\bf{1}_{D_{\sigma}}} a,$
where ${\bf{1_{D_{\sigma}}}}$ denotes the indicator of $D_{\sigma}.$
Given $S\subset \ell^2(\Z^d)$ a closed subspace  we define 
\[
S_{\sigma}=\{a^{\sigma}: a\in S\}, \quad \mbox{ for each } \sigma\in\mathcal{N}.
\] 
We define the class of approximating subspaces by
\begin{equation}\label{classD}
\mathcal{D}_{\mathcal{N}}^{\ell}=\{S\subset\ell^2(\Z^d)\colon S {\text{ is a subspace, }}  \dim(S)\le\ell \mbox{ and } S_{\sigma}\subseteq S,\;\; \forall \sigma\in\mathcal{N}\}.
\end{equation}

Note that $S \in \mathcal{D}_{\mathcal{N}}^{\ell}$  if and only if $ \dim(S)\le\ell$ and $S$ is the orthogonal sum of the subspaces $S_{\sigma}$ i. e.,
\[
S=\oplus_{\sigma\in \mathcal{N}}S_{\sigma}.
\]

For a given set  $\A= \{a_1,\dots, a_m\}\subset \ell^2(\Z^d)$ consider for each, $\sigma\in\mathcal{N},$ the Gramian matrix  $G_{\sigma}\in \C^{m\times m}$  of the data 
$\A_{\sigma} = \{a_1^{\sigma},\dots, a_m^{\sigma}\},$ that is
  $(G_{\sigma})_{k,l}= \langle a_k^{\sigma}, a_l^{\sigma},\rangle,\; k,l=1,\dots,m,$ 
  with eigenvalues $\la_1^{\si}\geq \dots,\geq\la_m^{\si},$ and orthonormal corresponding left eigenvectors $y_1^\si,\dots, y_m^\si.$
  
  Now set $\Lambda=\{\la_j^\si: j=1,\dots,m,\;\si\in\mathcal{N}\}$ and  collect in $\Lambda_{\ell}$ the   $\ell$ first biggest  eigenvalues of
  $\Lambda$ that is if $\la \in \Lambda_{\ell}$ then $\la \geq \mu$ for all $\mu \in \Lambda\setminus \Lambda_{\ell}.$ 

Write $\Lambda_{\ell}=\{\la_1,\dots,\la_{\ell}\}$. 
For each $s=1,\dots, \ell$ we define the sequence $q_s \in \ell^2(\Z^d)$ in the following way:

Since $\la_s= \lambda^{\si_s}_{j_s}$ for some $\si_s \in \mathcal{N}$ and some $j_s=1,\dots,m$, then
$\lambda_s$ is an eigenvalue of $G_{\si_s}$. Let $y_{j_s}^{\si_s}$ be the corresponding left eigenvector
$y_{j_s}^{\si_s}=(y_{j_s}^{\si_s}(1),\dots, y_{j_s}^{\si_s}(m)).$

Then define, if $\la_s \in \Lambda_{\ell}, \la_s \neq 0$ 
\begin{equation}\label{def-q}
q_{s} :=   (\la_s)^{-1/2} \sum_{k=1}^{m} y^{\si_s}_{j_s}(k) a_k^{\si_s}.
\end{equation}

If $\la_s=0$ we define $q_s$ to be the zero sequence.

With this notation we can state the main theorem of this section:

\begin{theorem}\label{PA}
Let $m, \ell\in\N$ and $\mathcal{N}$  a finite set. Assume that a set $\A= \{a_1,\dots, a_m\}\subset \ell^2(\Z^d)$ is given.
 
 Then there exists $S^*\in\mathcal{D}_{\mathcal{N}}^{\ell}$ that satisfies
\begin{equation}\label{eqPA} 
\sum_{j=1}^{m} \|a_j-P_{S^*} a_j\|^2\le \sum_{j=1}^{m} \|a_j-P_{S} a_j\|^2,  \quad \forall \,S\in\mathcal{D}_{\mathcal{N}}^{\ell}.
\end{equation}
Moreover, we have that

\begin{enumerate}
\item[(1)] $S^*=\mbox{span}\{q_1, \cdots, q_{\ell}\}$ where $q_1, \cdots, q_{\ell}$ are defined in \eqref{def-q}.
 Also, the vectors $\{q_1, \cdots, q_{\ell}\}$ form a Parseval frame for $S^*.$
\item[(2)] The error in the approximation is
\[
\mathcal{E}(\A, \mathcal{N}, \ell)= \sum_{\la \in \Lambda\setminus\Lambda_{\ell}}\lambda.
\]
\end{enumerate}
\end{theorem}

\begin{proof}
First, we observe that \eqref{eqPA} is equivalently to,
\[
\sum_{j=1}^{m} \|P_{S^*} a_j\|^2\ge \sum_{j=1}^{m} \|P_{S} a_j\|^2, \quad \forall\, S\in\mathcal{D}_{\mathcal{N}}^{\ell}.
\]

Furthermore, if $S\in\mathcal{D}_{\mathcal{N}}^{\ell}$ then

\begin{equation*}
\sum_{j=1}^{m} \|P_{S} a_j\|^2= \sum_{j=1}^{m} \|\sum_{\sigma\in\mathcal{N}}P_{S_{\sigma}} a_j\|^2
= \sum_{j=1}^{m} \|\sum_{\sigma\in\mathcal{N}}P_{S_{\sigma}} a_j^{\sigma}\|^2
= \sum_{\sigma\in\mathcal{N}} \sum_{j=1}^{m}\|P_{S_{\sigma}} a_j^{\sigma}\|^2, 
\end{equation*}
where $a_j= \sum_{\sigma\in\mathcal{N}} a_j^{\sigma}.$ 

In order to construct an optimal subspace $S^*\in\mathcal{D}_{\mathcal{N}}^{\ell}$ for the data $\A$,   we will find an optimal subspace  $S_{\sigma}$ for each $\sigma\in\mathcal{N}$, of dimension at most $\alpha_{\si}\in \{1,\dots,\ell\}$
for the data $\A_{\sigma} = \{a_1^{\sigma},\dots, a_m^{\sigma}\}$.
The existence of the optimal subspaces are provided by Theorem 4.1 of \cite{ACHM07}.
We need $\sum_{\sigma\in\mathcal{N}} {\text{dim}}(S_{\si}) = \sum_{\sigma\in\mathcal{N}} \alpha_{\sigma} = $dim$(S^*) \leq \ell$.

Thus  if $\mathcal{Q}=\{\alpha=\{\alpha_{\sigma}\} : 0\le \alpha_{\sigma}\le \ell {\text{ and }}  \sum_{\sigma\in\mathcal{N}} \alpha_{\sigma}\le \ell \},$ then for each choice of $\alpha \in\mathcal{Q}$ 
we will find optimal subspaces $\{S_{\sigma}^{\alpha}: \sigma\in\mathcal{N}\}$ and define 
$S^{\alpha}=\oplus_{\sigma\in\mathcal{N}} S_{\sigma}^{\alpha}.$

The candidate for $S^*$ is the space $S^{\alpha}$ which minimize the expression \eqref{eqPA} over all $\alpha \in \mathcal{Q}.$
Let $\beta \in \mathcal{Q}$ be the minimizer. Hence $\beta$ satisfies, 

\begin{equation}\label{max}
 \sum_{\sigma\in\mathcal{N}} \sum_{j=1}^{m}\|P_{S_{\sigma}^{\alpha}} a_j^{\sigma}\|^2 \leq
  \sum_{\sigma\in\mathcal{N}} \sum_{j=1}^{m}\|P_{S_{\sigma}^{\beta}} a_j^{\sigma}\|^2 \quad\forall \; \alpha \in \mathcal{Q}.
\end{equation}

Therefore, the subspace $S^*:= S^{\beta}=\oplus_{\sigma\in\mathcal{N}} S_{\sigma}^{\beta}$  is the optimal subspace
we need.
It is straightforward to see that $S^*\in\mathcal{D}_{\mathcal{N}}^{\ell}$  and that $S^*$ is optimal.

Using Theorem 4.1 of \cite{ACHM07} we obtain that for each $\alpha \in \mathcal{Q}$, the error of approximation for the data
$\A_{\sigma}$ and the class of subspaces of dimension at most $\alpha_{\si}$ is given by
$$
\mathcal{E}(\A_{\sigma},\alpha_{\si})=\sum_{s=\alpha_{\sigma}+1}^m \lambda_s^{\sigma}.
$$
So the distance between the $\alpha$-optimal  subspace $S^{\alpha}$ and the data $\A$ is,
\begin{equation}\label{err}
E(\alpha) = \sum_{\sigma\in\mathcal{N}} \mathcal{E}(\A_{\sigma},\alpha_{\si}) =\sum_{\sigma\in\mathcal{N}} \sum_{s=\alpha_{\sigma}+1}^{m} \lambda_s^{\sigma}.
\end{equation}

Let $\kappa$ be the number of elements in $\mathcal{N}$. We see that $E(\alpha)$ is minimum when the  $m\kappa-\ell$ eigenvalues used in  \eqref{err} are  the smallest
from the set $\Lambda=\{\la_j^\si: j=1,\dots,m,\;\si\in\mathcal{N}\}.$
Therefore if we set $\Lambda_{\ell}\subset \Lambda$ the set of the $\ell$ biggest eigenvalues from $\Lambda$,
 the optimal  $\beta=\{\beta_{\si}\} \in \mathcal{Q}$ satisfies that
 $$
\bigcup_{\si\in\mathcal{N}}\{\la^{\si}_1,\dots,\la^{\si}_{\beta_{\si}}\}=\Lambda_{\ell}.
$$
Therefore,
\[
\mathcal{E}(\A,\mathcal{N}, \ell) = \sum_{\sigma\in\mathcal{N}}\mathcal{E}(\A_{\sigma}, \beta_{\sigma})
= \sum_{\sigma\in\mathcal{N}} \sum_{j=1}^{m}\|a_j^{\sigma} - P_{S_{\sigma}^{\beta}} a_j^{\sigma}\|^2
= \sum_{\sigma\in\mathcal{N}} \sum_{s=\beta_{\sigma}+1}^{m} \lambda_s^{\sigma}= \sum_{\la \in \Lambda\setminus\Lambda_{\ell}}\lambda.
\]

\medskip

In order to construct the generators of $S^{\beta}$ it is enough to construct the generators of each $S^{\beta}_{\si}$.
Since  $S^{\beta}_{\si}$ are optimal subspaces for the data $\mathcal{A}_{\si} $
according with Theorem 4.1 of \cite{ACHM07} the generators of $S^{\beta}_{\si}$ are given by \eqref{def-q}. 
That is the set $\{q_s:  \si_s=\si\}$ is a Parseval frame of $S^{\beta}_{\si}$.

Since the subspaces $S^{\beta}_{\si}$ are mutually orthogonal,
$\{q_1,\dots,q_{\ell}\}$ is a set of Parseval frame generators for the optimal space $S^*=S^{\beta}.$ 
\end{proof}

\begin{remark}
As explained at the beginning of this section there is a  reason to consider this particular class of subspaces for the the discrete case. 
If a SIS V is $M$ extra-invariant for some proper subgroup $M$ of $\R^d$ containing $\Z^d$, 
then its fiber spaces $J_V(\omega)$ satisfies exactly the conditions that we imposed on the class 
$\mathcal{D}_{\mathcal{N}}^{\ell}$ where the partition of $\Z^d$ is $\{B_{\sigma}\cap \Z^d: \sigma \in \mathcal{N}\} $ and $B_{\sigma}$ and $\mathcal{N}$ are as  in \eqref{def-Bsigma}.
So,  the discrete result (Theorem \ref{PA}) provides a different proof of Theorem  \ref{solucion-PB}
using properties of range functions (\ref{U}),
without the need of Theorem \ref{Teorema-original}. Actually this proof  includes Theorem \ref{Teorema-original}. 
\end{remark}

\medskip

{\bf Acknowledgements.} We thank Ursula Molter and Victoria Paternostro for carefully reading the manuscript.

%%%%%%%%%%%%%%%%%%%%%%%%%%%%%
%%%%%%%%%%%%%%%%%%%%%%%%%%%%%
%%%%%%%%%%%%%%%%%%%%%

%\section{Approximations in the context of LCA groups}

%%%%%%%%%%%%%%%%%%%%%%%%%%%%
%
%%%%%%%%%%%%%%%%%%%%%%%%%%%%
%\bibliographystyle{abbrv}
%\bibliography{optimal}

\end{document}